\documentclass[sn-basic]{sn-jnl}

\usepackage{tikz}
\usetikzlibrary{intersections,calc,arrows.meta}

\usepackage[hang,small,bf]{caption}
\usepackage[subrefformat=parens]{subcaption}
\usepackage{ascmac}

\usepackage{pdfpages} 

\usepackage{bm}
\usepackage{amsmath}
\usepackage{mathrsfs}
\usepackage{amssymb}
\usepackage{booktabs}
\usepackage{comment}
\usepackage{mathtools}
\usepackage{multirow}
\usepackage{graphicx}
\usepackage{pgfplots}
\pgfplotsset{compat=1.18}
\numberwithin{equation}{section}

\usepackage{amsthm}
\newtheorem{mainthm}{Main Theorem}
\newtheorem{thm}{Theorem}[section]

\newtheorem{lem}{Lemma}[section]

\theoremstyle{definition}
\newtheorem{dfn}{Definition}[section]
\theoremstyle{definition}
\newtheorem{ex}{Example}[section]
\theoremstyle{definition}
\newtheorem{ass}{Assumption}

\theoremstyle{remark}
\newtheorem{rem}{Remark}[section]

\begin{document}

  \title[Article Title]{Real Log Canonical Thresholds at Non-singular Points}
  \author*{\fnm{Yuki} \sur{Kurumadani}}\email{kurumadani@sigmath.es.osaka-u.ac.jp}
  \affil*{\orgdiv{Graduate School of Engineering Science}, \orgname{Osaka University}, \orgaddress{\street{1 Chome-3 Machikaneyamacho}, \city{Toyonaka}, \postcode{560-0043}, \state{Osaka}, \country{Japan}}}

\abstract{Recent advances have clarified theoretical learning accuracy in Bayesian inference, revealing that the asymptotic behavior of metrics such as generalization loss and free energy, assessing predictive accuracy, is dictated by a rational number unique to each statistical model, termed the learning coefficient (real log canonical threshold) \citep{watanabe1}. For models meeting regularity conditions, their learning coefficients are known \citep{watanabe1}. However, for singular models not meeting these conditions, exact values of learning coefficients are provided for specific models like reduced-rank regression \citep{Aoyagi1}, but a broadly applicable calculation method for these learning coefficients in singular models remains elusive.

This paper extends the application range of the previous work and provides an approach that can be applied to many points within the set of realizable parameters. Specifically, it provides a formula for calculating the real log canonical threshold at many non-singular points within the set of realizable parameters. If this calculation can be performed, it is possible to obtain an upper bound for the learning coefficient of the statistical model. Thus, this approach can also be used to easily obtain an upper bound for the learning coefficients of statistical models. As an application example, it provides an upper bound for the learning coefficient of a mixed binomial model, and calculates the learning coefficient for a specific case of reduced-rank regression, confirming that the results are consistent with previous research.}


  \maketitle

\section{Introduction}
This paper extends the application range of \cite{kurumadani1}, using the same symbols and premises. Moreover, many theoretical aspects are based on \cite{kurumadani1}.

\subsection{Definitions of Symbols and Assumptions}
Consider a statistical model $p(x|\theta)$ with continuous parameters $\theta = (\theta_1, \ldots, \theta_d) \in \Theta (\subset \mathbb{R}^d) (d \geq 1)$, and let the true distribution be $q(x)$. Let $\chi$ be the set of possible data values $x$. Assume that the statistical model is \textit{realizable}, meaning there exists some parameter $\theta_*$ such that $q(x) = p(x|\theta_*)\ a.s.$. Such a parameter $\theta_*$ is called a realizable parameter, and the entire set of realizable parameters is denoted by $\Theta_*$. The prior distribution of the parameters $\varphi(\theta)$ satisfies $\varphi(\theta_*) > 0$ for \textit{any realizable parameter $\theta_*$}. Let $X$ be a random variable following the true distribution $q$, and let $\mathbb{E}_X[\cdot]$ denote the operation of taking the mean with respect to the random variable $X$. This paper assumes that the operations of taking expectations and partial derivatives with respect to $\theta$ are interchangeable.

The Kullback-Leibler divergence
\[
K(\theta) := \mathbb{E}_X\left[\log {\frac{p(X|\theta_*)}{p(X|\theta)}}\right]
\]
is assumed to be analytic around $\theta = \theta_*$. The log-likelihood ratio function
\[
f(x|\theta) := \log {\frac{p(x|\theta_*)}{p(x|\theta)}}
\]
is assumed to be $L^2$ and analytic around $\theta = \theta_*$. Note that if the model is realizable, then $\Theta_* = \{\theta \in \Theta | K(\theta) = 0\}$. For a fixed data $x \in \chi$, define the $m$-th order Taylor expansion of the log-likelihood ratio function $f(x|\theta)$ around $\theta = \theta_*$ with respect to $\theta_1, \ldots, \theta_s\ (s \leq d)$ as
\[
F_m(x|\theta_1, \ldots, \theta_s) := \sum_{\substack{i_1 + \cdots + i_s = m \\ i_1, \ldots, i_s \in \mathbb{Z}_{\geq 0}}} \frac{1}{i_1! \cdots i_s!} \times \left.\frac{\partial^m f(x|\theta)}{\partial \theta_1^{i_1} \cdots \partial \theta_s^{i_s}}\right|_{\theta = \theta_*} \times (\theta_1 - \theta_{1*})^{i_1} \cdots (\theta_s - \theta_{s*})^{i_s}.
\]

The learning coefficient is an important concept in determining the asymptotic behavior of generalization error and free energy. It is known to coincide with the concept of real log canonical threshold in algebraic geometry, which is defined using a method known as resolution of singularities. Here, resolution of singularities refers to the process described by Theorem~\ref{resolution thm}, which transforms an analytic function $F$ into a normal crossing.
\begin{thm}[Resolution of singularities]\label{resolution thm}
Let $F(x)$ be a real analytic function defined near the origin in $\mathbb{R}^d$ and not identically zero with $F(0) = 0$. Then there exists an open set $W (\subset \mathbb{R}^d)$ containing the origin, a real analytic manifold $U$, and a proper analytic map $g: U \rightarrow W$ such that:
\begin{itemize}
\item[(1)] Let $W_0 := F^{-1}(0), U_0 := g^{-1}(W_0)$. Then $g: U - U_0 \rightarrow W - W_0$ is an analytic isomorphism.
\item[(2)] At any point $Q \in U_0$, taking local coordinates $u = (u_1, \ldots, u_d)$ centered at $Q$, 
\begin{align}\label{eq:normalcrossing}
F(g(u)) = a(u)u_1^{k_1}u_2^{k_2}\cdots u_d^{k_d}\\
\left|g^{\prime}(u)\right| = \left|b(u)u_1^{h_1}u_2^{h_2}\cdots u_d^{h_d}\right|\notag
\end{align}
where $k_i, h_i (i=1,\ldots, d)$ are non-negative integers, and $a(u), b(u)$ are real analytic functions defined near the origin in $\mathbb{R}^d$, satisfying $a(0) \neq 0, b(0) \neq 0$.
\end{itemize}
\end{thm}

Expressions such as (\ref{eq:normalcrossing}) are referred to as \textit{normal crossings}.

\begin{dfn}[real log canonical threshold]\label{LambdaDef}\leavevmode\par
Let $F$ be a real analytic function defined on an open set $O$ in $\mathbb{R}^d$, and let $C$ be a compact set containing $O$. For each point $P$ in $C$ where $F(P)=0$, perform a coordinate transformation such that $P$ corresponds to the origin in $\mathbb{R}^d$. Apply Theorem~\ref{resolution thm}, and fix one $(W, U, g)$ as guaranteed by Theorem~\ref{resolution thm}(2). Also, denote the non-negative integers $h_i, k_i$ given by Theorem~\ref{resolution thm}(2) for any point $Q \in U_0$ as $h_i^{(Q)}, k_i^{(Q)}$.
\begin{itemize}
\item[(1)]
  Define the \textit{real log canonical threshold} $\lambda_P$ at point $P$ of function $F$ as:
\[
\lambda_P=\inf_{Q\in U_0}\left\{\min_{i=1,\ldots,d}{\frac{h_i^{(Q)}+1}{k_i^{(Q)}}}\right\}
\]
where if $k_i=0$, define $(h_i+1)/k_i=\infty$. It is known that this is well-defined, i.e., independent of the choice of $(W, U, g)$. \cite[Theorem 2.4]{watanabe1}
\item[(2)]
  Define the \textit{real log canonical threshold} $\lambda$ for the compact set $C$ of function $F$ as:
  \[
  \lambda = \inf_{P \in C}\lambda_P
  \]
  \cite[Definition 2.7]{watanabe1}
\item[(3)]
  In (2), for the point $P(\in C)$ that provides the minimum value, the term \textit{multiplicity} refers to the maximum number of $i$ satisfying $\lambda_P=(h_i^{(Q)}+1)/k_i^{(Q)}$. (If there are multiple points $P(\in C)$ giving the minimum value, multiplicity refers to the maximum of the maximum counts for each $i$.)
\end{itemize}
\end{dfn}

As already seen, this paper assumed a prior distribution $\varphi(\theta_*)>0$. In this context, it is known that the \textit{learning coefficient} $\lambda$ is equal to the real log canonical threshold for the compact set $\Theta_*=\{\theta\in\Theta|K(\theta)=0\}$ \citep{watanabe1}. That is, by performing the resolution of singularities at each point $P$ in $\Theta_*$ as guaranteed by Theorem~\ref{resolution thm}, and by moving $P$ across the entire $\Theta_*$, the minimum value of $\lambda_P$ matches the learning coefficient $\lambda$.

\subsection{Upper Bounds on the Learning Coefficients for General Models}
This paper assumed a prior distribution $\varphi(\theta_*)>0$. In this setting, if for any $v$ near $\theta=\theta_*=(u_*,v_*)$, $K(u_*,v)=0$ holds, it is known that:
\begin{equation}\label{eq:zyokai02}
\lambda \leq \frac{d_1}{2}
\end{equation}
\cite[Theorem 7.3]{watanabe1}

\subsection{An Example Where the Real Log Canonical Threshold Cannot be Computed Using \cite{kurumadani1}}
Consider a random variable $X$ following a binomial distribution Bin$(2,\theta)$ with parameter $\theta (0 < \theta < 1)$.
  \[
  \tilde{p}(X=x|\theta)=\binom{2}{x}\theta^x(1-\theta)^{2-x}=
  \begin{cases}
    (1-\theta)^2, & (x=0)\\
    2\theta(1-\theta), & (x=1)\\
    \theta^2, & (x=2)
  \end{cases}
  \]

  Consider the following mixed distribution model with parameters $(\theta_1, \theta_2, \tau)$:
  \begin{align}\label{eq_intro:eq01}
  p(X=x|\theta_1,\theta_2,\tau):=
  \left(\tau+\frac{1}{2}\right)\cdot\tilde{p}\left(X=x\middle|\theta_2+\frac{1}{2}\right)
  +\left(\frac{1}{2}-\tau\right)\cdot\tilde{p}\left(X=x\middle|\theta_1+\frac{1}{2}\right)&\\
  (x=0,1,2)\notag&
  \end{align}
  
  Assume the true distribution is $\tilde{p}(X|1/2)$, then this model realizes the true distribution at the parameters $(\theta_1, \theta_2, \tau) = 0$.
  Let us compute the real log canonical threshold at the origin.

  First, we attempt to compute the real log canonical threshold using the methods described in \cite{kurumadani1}.
  The first derivatives of the log-likelihood ratio function $f$ with respect to each parameter are as follows:
  \begin{gather*}
    \left.\frac{\partial f}{\partial\theta_1}\right|_{(\theta_1,\theta_2,\tau)=0} = 2-2x, ~~ 
    \left.\frac{\partial f}{\partial\theta_2}\right|_{(\theta_1,\theta_2,\tau)=0} = 2-2x,~~
    \left.\frac{\partial f}{\partial\tau}\right|_{(\theta_1,\theta_2,\tau)=0} = 0
  \end{gather*}

  There is a first-order dependency relationship:
  \begin{align*}
  \left.\frac{\partial f}{\partial\theta_2}\right|_{(\theta_1,\theta_2,\tau)=0}
  = \left.\frac{\partial f}{\partial\theta_1}\right|_{(\theta_1,\theta_2,\tau)=0}
  \end{align*}
  If we perform the coordinate transformation $\theta_1^{\prime} = \theta_1+\theta_2$, it eliminates the first derivative with respect to $\theta_2$:
  \begin{gather*}
    \left.\frac{\partial f}{\partial\theta_1^{\prime}}\right|_{(\theta_1^{\prime},\theta_2,\tau)=0} = 2-2x, ~~ 
    \left.\frac{\partial f}{\partial\theta_2}\right|_{(\theta_1^{\prime},\theta_2,\tau)=0} = 0,~~
    \left.\frac{\partial f}{\partial\tau}\right|_{(\theta_1^{\prime},\theta_2,\tau)=0} = 0\\
    \left.\frac{\partial^2 f}{\partial\theta_2^2}\right|_{(\theta_1^{\prime},\theta_2,\tau)=0} = -8(1-4x+2x^2),~~
    \left.\frac{\partial^2 f}{\partial\theta_2\partial\tau}\right|_{(\theta_1^{\prime},\theta_2,\tau)=0} = 8-8x,~~
    \left.\frac{\partial^2 f}{\partial\tau^2}\right|_{(\theta_1^{\prime},\theta_2,\tau)=0} = 0
  \end{gather*}  
  Furthermore, a first-order dependency relationship exists:
  \[
  \left.\frac{\partial^2 f}{\partial\theta_2\partial\tau}\right|_{(\theta_1^{\prime},\theta_2,\tau)=0} =
  4\left.\frac{\partial f}{\partial\theta_1^{\prime}}\right|_{(\theta_1^{\prime},\theta_2,\tau)=0}
  \]
  Using the coordinate transformation $\theta_1^{\prime\prime} = \theta_1^{\prime}+4\theta_2\tau$, we achieve:
  \begin{gather*}
    \left.\frac{\partial f}{\partial\theta_1^{\prime\prime}}\right|_{(\theta_1^{\prime\prime},\theta_2,\tau)=0} = 2-2x, ~~ 
    \left.\frac{\partial f}{\partial\theta_2}\right|_{(\theta_1^{\prime\prime},\theta_2,\tau)=0} = 0,~~
    \left.\frac{\partial f}{\partial\tau}\right|_{(\theta_1^{\prime\prime},\theta_2,\tau)=0} = 0\\
    \left.\frac{\partial^2 f}{\partial\theta_2^2}\right|_{(\theta_1^{\prime\prime},\theta_2,\tau)=0} = -8(1-4x+2x^2),~~
    \left.\frac{\partial^2 f}{\partial\theta_2\partial\tau}\right|_{(\theta_1^{\prime\prime},\theta_2,\tau)=0} = 0,~~
    \left.\frac{\partial^2 f}{\partial\tau^2}\right|_{(\theta_1^{\prime\prime},\theta_2,\tau)=0} = 0
  \end{gather*}  
  Clearly, the two random variables $2-2X$ and $-8(1-4X+2X^2)$ are linearly independent, satisfying \cite[Assumption 1]{kurumadani1} for $m=2$. However,
  \[
  F_2(x|\theta_2,\tau) = -4(1-4x+2x^2)\theta_2^2
  \]
  leads to $F_2(x|\theta_2,\tau) = 0$ having non-trivial solutions $(\theta_2,\tau) = (0, \forall\tau)$.
  Therefore, applying \cite[Main Theorem 2]{kurumadani1} directly to obtain the real log canonical threshold is not possible.

\subsection{Technique of Treating Some Parameters as Constants} \label{ch_ex}
Having decided not to derive the real log canonical threshold using the method from \cite{kurumadani1}, this paper introduces the principal idea of \textit{treating the parameter $\tau$ as a constant} to derive the threshold. By fixing $\tau$ and considering only the parameters $(\theta_1, \theta_2)$ in the statistical model (\ref{eq_intro:eq01}), the derivatives are computed as follows:
  \begin{gather*}
    \left.\frac{\partial f}{\partial\theta_1}\right|_{(\theta_1, \theta_2)=0} = 2(1-2\tau)(1-x), ~~ 
    \left.\frac{\partial f}{\partial\theta_2}\right|_{(\theta_1, \theta_2)=0} = 2(1+2\tau)(1-x)
  \end{gather*}  
  These exhibit a linear dependency relationship for each fixed $\tau$:
  \footnote{More specifically, if $L$ represents the set of analytic functions defined near $\tau=0$, then they are linearly dependent over $L$.}
  \[
  \left.\frac{\partial f}{\partial\theta_2}\right|_{(\theta_1, \theta_2)=0} = 
  \frac{1+2\tau}{1-2\tau}\cdot\left.\frac{\partial f}{\partial\theta_1}\right|_{(\theta_1, \theta_2)=0}
  \]
  Applying the coordinate transformation
  \begin{equation}\label{eq_intro:eq07}
  \theta_1^{\prime} = \theta_1+\frac{1+2\tau}{1-2\tau}\cdot\theta_2
  \end{equation}
  results in:
  \begin{align}
    \left.\frac{\partial f}{\partial\theta_1^{\prime}}\right|_{(\theta_1^{\prime}, \theta_2)=0} &= 2(1-2\tau)(1-x)\label{eq_intro:eq02}\\
    \left.\frac{\partial f}{\partial\theta_2}\right|_{(\theta_1^{\prime}, \theta_2)=0} &= 0\label{eq_intro:eq03}\\
    \left.\frac{\partial^2 f}{\partial\theta_2^2}\right|_{(\theta_1^{\prime}, \theta_2)=0} &= -\frac{8(1+2\tau)}{1-2\tau}(1-4x+2x^2)\label{eq_intro:eq04}
  \end{align}    
  The terms (\ref{eq_intro:eq02}) and (\ref{eq_intro:eq04}) are verified to be linearly independent over $\mathbb{R}$ when $\tau=0$. Each $\tau$ in any small neighborhood around $\tau=0$, substituting any $\tau \in U$, ensures that (\ref{eq_intro:eq02}) and (\ref{eq_intro:eq04}) remain linearly independent over $\mathbb{R}$. Therefore, treating only $(\theta_1^{\prime}, \theta_2)$ as parameters, this approach fulfills the same assumptions as \cite[Assumption 1]{kurumadani1} for $r=1, d=2, m=2$.

  Consequently, with $\tau$ fixed, it appears feasible to compute the real log canonical threshold using \cite[Main Theorem 2]{kurumadani1}. Indeed, the blow-up described in \cite[Main Theorem 2]{kurumadani1}, namely, the blow-up centered around the submanifold $(\theta_1^{\prime}, \theta_2)=0$, demonstrates that a normal crossing of $K(\theta)$ can be obtained. Let us confirm this. Hereafter, $\theta_1^{\prime}$ will be denoted as $\theta_1$.

Using Mathematica, we computed the Taylor expansion of $K(\theta)$ as follows:
\begin{align}\label{eq_intro:eq05}
  K(\theta)
  &= (1-2\tau)^2\theta_1^2 + 8\tau^2(1-2\tau)^2\theta_1^4 + \frac{8(1+2\tau)^2}{(1-2\tau)^2}\theta_2^4 \notag \\
  &\quad - 16\tau(1-4\tau^2)\theta_1^3\theta_2 + 8(1+6\tau+8\tau^2)\theta_1^2\theta_2^2 \notag \\
  &\quad - \frac{16(1+2\tau)^2}{1-2\tau}\theta_1\theta_2^3 + (\text{terms of order 5 or higher}) \notag \\
  &= (1-2\tau)^2\theta_1^2 + \frac{8(1+2\tau)^2}{(1-2\tau)^2}\theta_2^4 + \text{(higher order terms)}.
\end{align}

Using (\ref{eq_intro:eq02}) and (\ref{eq_intro:eq04}), we have
  \begin{align}
    \frac{1}{2}\mathbb{E}_X\left[\left\{F_1(X|\theta_1)+F_2(X|\theta_2)\right\}^2\right]
    &=\frac{1}{2}\mathbb{E}_X\left[\left\{
    2(1-2\tau)(1-X)\theta_1-\frac{4(1+2\tau)}{1-2\tau}(1-4X+2X^2)\theta_2^2
    \right\}^2\right]\notag\\
    &=(1-2\tau)^2\theta_1^2+\frac{8(1+2\tau)^2}{(1-2\tau)^2}\theta_2^4\label{eq_intro:eq08}
  \end{align}
  which corresponds to the lower order terms of (\ref{eq_intro:eq05}).
  Thus, it has been confirmed that \cite[Main Theorem 1]{kurumadani1} holds.

Next, let's consider a blow-up $g_1$ at $(\theta_1, \theta_2) = 0$ and calculate the real log canonical threshold.

\begin{itemize}
  \item [(a)]
  First, by transforming (\ref{eq_intro:eq05}) with $\theta_2 = \theta_1 \theta_2'$, and using power series $h_1, a_1$, we have:
  \[
  K(\theta) = \theta_1^2 \left\{
    (1-2\tau)^2 + \frac{8(1+2\tau)^2}{(1-2\tau)^2} \theta_1^2 \theta_2'^4 + \theta_1 h_1(\theta_1, \theta_2')
  \right\} = \theta_1^2 a_1(\theta_1, \theta_2')
  \]
  At any point on $g_1^{-1}(0) = \{(\theta_1, \theta_2') | \theta_1 = 0\}$:
  \[
  \forall \theta_2',~ a_1(0, \theta_2') = (1-2\tau)^2 \neq 0
  \]
  (for $\tau$ near 0), confirming normal crossings in the local coordinates $(\theta_1, \theta_2')$.

  \item[(b)]
  Next, transforming (\ref{eq_intro:eq05}) with $\theta_1 = \theta_2 \theta_1'$, and using power series $h_2, a_2$, we obtain:
  \[
  K(\theta) = \theta_2^2 \left\{
    (1-2\tau)^2 \theta_1'^2 + \frac{8(1+2\tau)^2}{(1-2\tau)^2} \theta_2^2 + \theta_2 h_2(\theta_1', \theta_2)
  \right\} = \theta_2^2 a_2(\theta_1', \theta_2)
  \]
  At any point on $g_1^{-1}(0) = \{(\theta_1', \theta_2) | \theta_2 = 0\}$:
  \[
  a_2(\theta_1', 0) = 0 \Leftrightarrow \theta_1' = 0
  \]
  (for $\tau$ near 0), achieving normal crossings at points other than $(\theta_1', \theta_2) = 0$.
\end{itemize}

Therefore, it is necessary to find normal crossings at the point $(\theta_1', \theta_2) = 0$. Below, we perform another blow-up, $g_2$, centered at the point $(\theta_1', \theta_2) = 0$.

\begin{itemize}
  \item[(b1)]
  First, by transforming with $\theta_2 = \theta_1' \theta_2''$, and using power series $h_3, a_3$, we have:
  \begin{align*}
  K(\theta)
  &= \theta_1'^4 \theta_2''^2 \left\{
  (1-2\tau)^2 + \frac{8(1+2\tau)^2}{(1-2\tau)^2} \theta_2''^2 + \theta_1' h_3(\theta_1', \theta_2'')
  \right\} \\
  &= \theta_1'^4 \theta_2''^2 a_3(\theta_1', \theta_2'')    
  \end{align*}
  At any point on $g_2^{-1}(0) = \{(\theta_1', \theta_2'') | \theta_1' = 0\}$, for all $\theta_2''$:
  \[
  a_3(0, \theta_2'') = (1-2\tau)^2 + \frac{8(1+2\tau)^2}{(1-2\tau)^2} \theta_2''^2 \neq 0
  \]
  (for $\tau$ near 0), indicating that normal crossings are achieved in the local coordinates $(\theta_1', \theta_2'')$.

  \item[(b2)]
  Next, by transforming with $\theta_1' = \theta_2 \theta_1''$, and using power series $h_4, a_4$, we get:
  \[
  K(\theta)
  = \theta_2^4 \left\{
  (1-2\tau)^2 \theta_1''^2 + \frac{8(1+2\tau)^2}{(1-2\tau)^2} + \theta_2 h_4(\theta_1'', \theta_2)
  \right\}
  = \theta_2^4 a_4(\theta_1'', \theta_2)
  \]
  At any point on $g_2^{-1}(0) = \{(\theta_1'', \theta_2) | \theta_2 = 0\}$, for all $\theta_1''$:
  \[
  a_4(\theta_1'', 0) = (1-2\tau)^2 \theta_1''^2 + \frac{8(1+2\tau)^2}{(1-2\tau)^2} \neq 0
  \]
  (for $\tau$ near 0), indicating that normal crossings are achieved in this local coordinate system.
\end{itemize}

The summary of normal crossings for each local coordinate is as follows (refer to symbols $k_i^{(Q)}, h_i^{(Q)}$ in Definition~\ref{LambdaDef}).

\begin{table}[htbp]
  \centering
  \caption{Normal Crossings for Each Local Coordinate}
  \begin{tabular}{ccccccc}
    \hline
    No. & Local Coord. & $K(\theta)$ & Jacobian & $(k_1^{(Q)}, k_2^{(Q)})$ & $(h_1^{(Q)}, h_2^{(Q)})$ & $\lambda$ \\
    \hline
    (a) & $(\theta_1, \theta_2')$ & 
 $\theta_1^2 a_1(\theta_1, \theta_2')$ & $\theta_1$ & $(2,0)$ & $(1,0)$ & 1 \\
    (b) & $(\theta_1', \theta_2) \neq 0$ & 
 $\theta_2^2 a_2(\theta_1', \theta_2)$ & $\theta_2$ & $(0,2)$ & $(0,1)$ & 1 \\
    (b1) & $(\theta_1', \theta_2'')$ & 
 $\theta_1'^4 \theta_2''^2 a_3(\theta_1', \theta_2'')$ & $\theta_1'^2 \theta_2''$ & $(4,2)$ & $(2,1)$ & 3/4 \\
    (b2) & $(\theta_1'', \theta_2)$ & 
 $\theta_2^4 a_4(\theta_1'', \theta_2)$ & $\theta_2^2$ & $(0,4)$ & $(0,2)$ & 3/4 \\
    \hline
  \end{tabular}
\end{table}

From the above, it is evident that the real log canonical threshold is $3/4$ (multiplicity is 1).

Thus, by treating the parameter $\tau$ as a constant, it was possible to conduct the same discussion for the parameters $(\theta_1, \theta_2)$ as in \cite[Main Theorem 2]{kurumadani1}, and obtain the real log canonical threshold.

\subsection{Why Can the Real Log Canonical Threshold Be Computed by Treating Some Parameters as Constants?}
Let's consider why treating the parameter $\tau$ as a constant facilitates the discussion in the example above. The essential point is that the log-likelihood ratio function $f$ satisfies the following for parameter $\tau$:
\begin{equation}\label{eq_intro:eq06}
  f(\cdot|\theta=0,\forall\tau)=0~~\text{a.s.}
\end{equation}
In terms of the statistical model $p$, this means that $p(\cdot|\theta=0,\forall\tau)=q(\cdot)$ (a.s.), i.e., \textit{the true distribution $q$ can be realized with parameters $\theta$ alone, without setting $\tau=0$}. Indeed, in the example above, setting parameters $(\theta_1, \theta_2)=0$ in (\ref{eq_intro:eq01}), we have:
\[
p(\cdot|\theta_1=0,\theta_2=0,\forall\tau)=
\left(\tau+\frac{1}{2}\right)\cdot\tilde{p}\left(\cdot~\middle|0+\frac{1}{2}\right)
+\left(\frac{1}{2}-\tau\right)\cdot\tilde{p}\left(\cdot~\middle|0+\frac{1}{2}\right)
=\tilde{p}\left(\cdot~\middle|\frac{1}{2}\right)=q(\cdot)
\]
This property holds true. By taking the expectation in (\ref{eq_intro:eq06}), the Kullback-Leibler information $K(\theta,\tau)$ satisfies:
\begin{equation}\label{eq_intro:eq09}
  K(\theta=0,\forall \tau)=0
\end{equation}
This implies that when treating $\tau$ as a constant, no constant term (i.e., terms involving only $\tau$) arises in the Taylor expansion of $K(\theta)$ at $\theta=0$. Indeed, in the example above, no constant term appeared in (\ref{eq_intro:eq05}).

Geometrically interpreting condition (\ref{eq_intro:eq09}), the set of realizable parameters $\Theta_*$ near the origin extends along the $\tau$-axis as a line. This means that in any neighborhood of the origin, there exist realizable parameters other than the origin, making it impossible to achieve a normal crossing with blow-ups centered at the origin. From this, it is expected that blow-ups centered not at the origin but along the line $\{(\theta_1, \theta_2, \tau) | \theta_1 = \theta_2 = 0\}$ would be effective, as confirmed above where a normal crossing was achieved through a blow-up centered on this line. This demonstrates the effectiveness of applying the method of \cite{kurumadani1} by treating parameter $\tau$ as a constant.

Generalizing this example, this paper extends the results of \cite{kurumadani1} to "cases where some parameters are realizable." As seen in the example, this approach significantly broadens the range of statistical models for which the real log canonical threshold can be calculated. Specifically, if a point $\theta_*$ in the set of realizable parameters $\Theta_*$ is non-singular, it necessarily satisfies (\ref{eq_intro:eq09}), as discussed later. As is well known, most points in the set of realizable parameters $\Theta_*$ are non-singular (for example, see \cite[Remark 2.4]{watanabe1}). Therefore, by using the formula for the real log canonical threshold generalized from the results of \cite{kurumadani1} for non-singular points, it is possible to obtain an upper bound for the learning coefficient. This is the main point of this paper.

\newpage
\section{Main Theorem}
In this paper, we consider a statistical model that has been translated such that the realizable parameter $\theta_*$ is at the origin, without loss of generality.

Let $d_1$ be an integer greater than or equal to one, and $d_2$ a non-negative integer, and consider a statistical model $p(x|\theta,\tau)$ with $d_1 + d_2$ parameters $(\theta_1,\ldots,\theta_{d_1},\tau_1,\ldots,\tau_{d_2})$.

\subsection{Feasibility by Some Parameters}

\begin{dfn}\label{def:yuugen}\leavevmode\par
  A statistical model $p(x|\theta,\tau)$ with parameters $(\theta,\tau)=(\theta_1,\ldots,\theta_{d_1},\tau_1,\ldots,\tau_{d_2})$ is said to be \textit{realizable by $\theta=0$ alone}\footnote{This term is specific to this paper and not generally used.} at $(\theta,\tau)=0$, if for a sufficiently small neighborhood $U$ of $\tau=0$, it holds identically that
  \begin{equation}\label{eq_def:yuugen01}
  \forall \tau\in U,~~ p(\cdot|\theta=0,\tau)=q(\cdot)~~ \text{a.s.}
  \end{equation}
\end{dfn}

\begin{lem}[Characterization of Realizability]\label{lem:zitugen}\leavevmode\par
  Let $U$ be a sufficiently small neighborhood around $\tau=0$. The following three conditions are equivalent:
  \begin{itemize}
    \item [(1)] (\ref{eq_def:yuugen01}), meaning
      \[
      \forall \tau\in U,~~ p(\cdot|\theta=0,\tau)=q(\cdot)~~ \text{a.s.}
      \]
    \item [(2)] 
      \begin{equation}\label{eq_lem:zitugen01}
        \forall \tau\in U,~~ f(\cdot|\theta=0,\tau)=0~~ \text{a.s.}
      \end{equation}
    \item [(3)]
      \[
      \forall \tau\in U,~~ K(\theta=0,\tau)=0
      \]
  \end{itemize}
  Therefore, if condition (\ref{eq_def:yuugen01}) is met, the constant term in the Taylor expansion of $f$ and $K$ at $\theta=0$, when treating $\tau$ as a constant, is zero.
\end{lem}

\begin{proof}\leavevmode\par
  \[
  f(\cdot|\theta=0,\tau)=\log {\frac{q(\cdot)}{p(\cdot|\theta=0,\tau)}}
  \]
  Hence, the equivalence of (1) and (2) is obvious, and $(2)\Rightarrow(3)$ follows by taking expectations. Therefore, we only need to show $(3)\Rightarrow(2)$. In this paper, we assume realizability for the statistical model $p(x|\theta,\tau)$ in the usual sense, which implies
  \[
  K(\theta,\tau)=\mathbb{E}_X\left[f(X|\theta,\tau)\right]\geq\frac{1}{2}\mathbb{E}_X\left[f^2(X|\theta,\tau)\right]
  \]
  according to \cite[Theorem 6.3]{watanabe1}. Setting $\theta=0$ and considering
  \[
  \forall \tau\in U,~~ 0= K(\theta=0,\tau)\geq \frac{1}{2}\mathbb{E}_X\left[f^2(X|\theta=0,\tau)\right]
  \]
  yields $f(X|\theta=0,\forall\tau)=0$ (a.s.).
\end{proof}

\subsection{Main Theorem}

When treating $\tau$ as a constant, Assumption 1 from \cite{kurumadani1} can be reformulated as follows:

\begin{ass}\label{mainass}\leavevmode\par
Assume the following (0)-(3) with $\tau$ as any constant in a sufficiently small neighborhood $U$ of $\tau=0$:
\begin{itemize}
  \item[(0)] Feasibility holds only for $\theta=0$, namely $f(\cdot|\theta=0,\tau)=0$ (a.s.). We also assume that in the case of $d_2=0$, condition (0) is always satisfied.

  \item[(1)] For $r$ parameters $\theta_{1},\ldots,\theta_{r}$, the $r$ random variables
    \[
    \left.\frac{\partial f(X|\theta,\tau)}{\partial\theta_{1}}\right|_{(\theta_1,\ldots,\theta_{d_1})=0}, \ldots, \left.\frac{\partial f(X|\theta,\tau)}{\partial\theta_{r}}\right|_{(\theta_1,\ldots,\theta_{d_1})=0}
    \]
    are linearly independent at $\tau=0$, i.e.,
    \begin{equation}\label{eq_mainass2_01}
    \left.\frac{\partial f(X|\theta,\tau)}{\partial\theta_{1}}\right|_{(\theta_1,\ldots,\theta_{d_1},\tau_1,\ldots,\tau_{d_2})=0}, \ldots, \left.\frac{\partial f(X|\theta,\tau)}{\partial\theta_{r}}\right|_{(\theta_1,\ldots,\theta_{d_1},\tau_1,\ldots,\tau_{d_2})=0}
    \end{equation}
    are linearly independent.
  \item[(2)] For $d_1-r$ parameters $\theta_{r+1},\ldots,\theta_{d_1}$ and an integer $m \geq 1$, derivatives of the log-likelihood ratio function $f(X|\theta,\tau)$ up to $m-1$ with respect to $\theta_{r+1},\ldots,\theta_{d_1}$ are zero with probability one at $(\theta_1,\ldots,\theta_{d_1})=0$, regardless of $\tau$, i.e.,
    \[
    \forall\tau\in U,~~ F_1(X|\theta_{r+1},\ldots,\theta_{d_1},\tau)=\cdots=F_{m-1}(X|\theta_{r+1},\ldots,\theta_{d_1},\tau)=0~~\text{a.s.}
    \]
    The largest such $m$ is redefined as $m$. For convenience, if $m=1$, then $r=d_1$ is assumed.
  \item[(3)] For $m \geq 2$, when \textbf{$\tau = 0$}, for each $(\theta_{r+1},\ldots,\theta_{d_1}) \neq 0$ at $\tau=0$, one of the following holds:
    \begin{itemize}
      \item [(i)] $\left.F_m(X|\theta_{r+1},\ldots,\theta_{d_1})\right|_{\tau=0}=0$ (a.s.)
      \item [(ii)] $\left.F_m(X|\theta_{r+1},\ldots,\theta_{d_1})\right|_{\tau=0}$ and the $r$ random variables from (\ref{eq_mainass2_01}) are linearly independent.
    \end{itemize}
    If $m=1$, (3)(ii) is always satisfied for convenience.
\end{itemize}
\end{ass}

\begin{rem}\label{rem_mainass}\leavevmode\par
  When $d_2=0$, Assumption~\ref{mainass} is the same as \cite[Assumption 1]{kurumadani1}, thus generalizing \cite[Assumption 1]{kurumadani1}.
\end{rem}

\begin{rem}\label{exmainthm}\leavevmode\par
  For the example discussed in Section~\ref{ch_ex}, after performing the variable transformation (\ref{eq_intro:eq07}), it is verified that Assumption~\ref{mainass} is satisfied for $(d_1,r,m)=(2,1,2)$, with $(\theta_1^{\prime},\theta_2)$ corresponding to $(\theta_1, \theta_2)$ and $\tau$ in Assumption~\ref{mainass}.

  First, Assumption~\ref{mainass}(0) is already verified in (\ref{eq_intro:eq06}). Regarding Assumption~\ref{mainass}(1), for parameter $\theta_1^{\prime}$, by substituting $\tau=0$ into (\ref{eq_intro:eq02}),
  \begin{equation}\label{eq_exmainthm01}
  \left.\frac{\partial f}{\partial\theta_1^{\prime}}\right|_{(\theta_1^{\prime},\theta_2,\tau)=0}=2(1-X)
  \end{equation}
  shows that it is non-zero and thus linearly independent.

  Finally, Assumption~\ref{mainass}(2) and (3) are satisfied since for parameter $\theta_2$, from (\ref{eq_intro:eq03}) regardless of $\tau$,
  \[
  \left.\frac{\partial f}{\partial\theta_2}\right|_{(\theta_1^{\prime},\theta_2)=0}=0
  \]
  meets Assumption~\ref{mainass}(2). Moreover, substituting $\tau=0$ into (\ref{eq_intro:eq04}),
  \[
  \left.F_2(X|\theta_2)\right|_{\tau=0}
  =\frac{1}{2}\left.\frac{\partial^2 f}{\partial\theta_2^2}\right|_{(\theta_1^{\prime},\theta_2,\tau)=0}\theta_2^2
  =-4(1-4X+2X^2)\theta_2^2
  \]
  confirms that it is linearly independent from (\ref{eq_exmainthm01}) for any $\theta_2\neq 0$, thus fulfilling Assumption~\ref{mainass}(3)(ii).

  This confirms that Assumption~\ref{mainass} is satisfied for $(d_1,r,m)=(2,1,2)$.\\
\end{rem}

If Assumption~\ref{mainass} is satisfied, even if only some parameters realize the true distribution, the main results of the paper \cite{kurumadani1}, \cite[Main Theorem 1]{kurumadani1} and \cite[Main Theorem 2]{kurumadani1}, are still applicable. The proof will be provided in the next section.

\begin{mainthm}\label{mainthm}\leavevmode\par
Assume that the statistical model $p(x|\theta,\tau)$ with parameters $(\theta,\tau)$ satisfies Assumption~\ref{mainass}(0)(1)(2). Then, treating $\tau$ as any constant in a neighborhood of $\tau=0$, the Taylor expansion of $K(\theta,\tau)$ at $\theta=0$ can be expressed as:
\begin{equation}\label{eq_mainthm01}
K(\theta,\tau)=\frac{1}{2}\mathbb{E}_X\left[\left\{F_1(X|\theta_1,\ldots,\theta_{r})+F_m(X|\theta_{r+1},\ldots,\theta_{d_1})\right\}^2\right]+\text{(higher order terms)}
\end{equation}
where $\tau$ may be included in all terms on the right side. Additionally, (higher order terms) do not include terms involving:
\begin{itemize}
  \item Terms of degree up to $2m$ involving only $\theta_{r+1},\ldots,\theta_{d_1}$
  \item Terms of first degree in $\theta_1,\ldots,\theta_r$ and up to $m$th degree in $\theta_{r+1},\ldots,\theta_{d_1}$
  \item Second-degree terms involving only $\theta_1,\ldots,\theta_r$
\end{itemize}

\end{mainthm}

\begin{rem}\leavevmode\par
  It has already been confirmed in (\ref{eq_intro:eq05}) and (\ref{eq_intro:eq08}) that Main Theorem~\ref{mainthm} holds for the example discussed in Section~\ref{ch_ex}.
\end{rem}

\begin{mainthm}\label{mainthm1}\leavevmode\par
  Consider a statistical model $p(x|\theta,\tau)$ that satisfies Assumption~\ref{mainass}(0)-(3). Consider the following blow-up $g$ centered at the subvariety $W_0:=\{(\theta,\tau)|\theta=0\} \subset \mathbb{R}^{d_1+d_2}$:
  \begin{itemize}
  \item [(a)] Perform one blow-up centered at $W_0$.
  \item [(b)] If the exceptional surface from (a) is $\{\theta_i=0\}$ (where $i=r+1,\ldots,d_1$), perform another blow-up centered at the subvariety $\{(\theta,\tau)|\theta_1=\cdots=\theta_r=\theta_i=0\}$.
  \item [(c)] If the exceptional surface from (b) is $\{\theta_i=0\}$, repeat (b) until the total number of blow-ups reaches $m$.
  \end{itemize}

  When the $m$-th blow-up results in the exceptional surface $\{\theta_i=0\}$, $g=g_i$ can be expressed as follows (for $i=r+1,\ldots,d_1$):
  \[
  g_i\colon(\theta^{\prime}_1,\ldots,\theta^{\prime}_{i-1},\theta_i,\theta^{\prime}_{i+1},\ldots,\theta^{\prime}_{d_1})
  \mapsto
  (\theta_1,\ldots,\theta_{i-1},\theta_i,\theta_{i+1},\ldots,\theta_{d_1});
  \]
  \[
  \theta_1=\theta^m_i\theta_1^{\prime},\ldots,\theta_r=\theta^m_i\theta_r^{\prime},\\
  \theta_{r+1}=\theta_i\theta_{r+1}^{\prime},\ldots,\theta_{i-1}=\theta_i\theta_{i-1}^{\prime},\ 
  \theta_{i+1}=\theta_i\theta_{i+1}^{\prime},\ldots,\theta_{d_1}=\theta_i\theta_{d_1}^{\prime}
  \]

  Then, for a subset $S$ of $U_0:=g^{-1}(W_0)$ in local coordinates $(\theta^{\prime}_1,\ldots,\theta^{\prime}_{i-1},\theta_i,\theta^{\prime}_{i+1},\ldots,\theta^{\prime}_{d_1})$:
  \[
  S:=\bigcup_{i=r+1}^{d_1}
  \left\{
  (\theta^{\prime}_1,\ldots,\theta^{\prime}_{d_1},\tau)\middle|
  \begin{array}{l}
    (\theta^{\prime}_1,\ldots,\theta^{\prime}_{r},\theta_i)=0,~ \tau=0 \\
    F_m(X|\theta^{\prime}_{r+1},\ldots,\theta^{\prime}_{i-1},1,\theta^{\prime}_{i+1},\ldots,\theta^{\prime}_{d_1})= 0 ~\text{a.s.}
  \end{array}
  \right\}\subset U_0 
  \]
  A normal crossing of $K(\theta,\tau)$ is obtained at points in $U_0$ not belonging to $S$, and:
  \[
  \inf_{Q\in U_0\setminus S}\left\{\min_{i=1,\ldots,{d_1}}{\frac{h_i^{(Q)}+1}{k_i^{(Q)}}}\right\}=\frac{d_1-r+rm}{2m}
  \]
  is satisfied (multiplicity is 1). For symbols $k_i^{(Q)},h_i^{(Q)} $ refer to Definition~\ref{LambdaDef}.

  Especially when all parameters $(\theta_{r+1},\ldots,\theta_{d_1})\neq0$ satisfy Assumption~\ref{mainass}(3)(ii), the real log canonical threshold $\lambda_O$ at the origin $O$ is given by:
  \begin{equation}\label{eq:mainthm01}
  \lambda_O=\frac{d_1-r+rm}{2m}
  \end{equation}
  In this case, the set of points in the parameter space $\Theta$ near the origin satisfying $K(\theta,\tau)=0$ is $W_0$.
\end{mainthm}

\begin{rem}\leavevmode\par
  It has been confirmed that Main Theorem~\ref{mainthm1} holds for the example discussed in Section~\ref{ch_ex}, as already noted in Remark~\ref{exmainthm}, where $(d_1,r,m)=(2,1,2)$ satisfies Assumption~\ref{mainass} and every parameter $\theta_2\neq 0$ meets Assumption~\ref{mainass}(3)(ii). Therefore, the real log canonical threshold at the origin of $K(\theta_1,\theta_2,\tau)$ should be:
  \[
  \lambda_O=\frac{2-1+2}{2\cdot 2}=\frac{3}{4}
  \]
  (multiplicity is 1), which indeed aligns with the results derived in Section~\ref{ch_ex}. The center of blow-ups to achieve this also matches $W_0$ as outlined in Main Theorem~\ref{mainthm1}.
\end{rem}

\begin{rem}\leavevmode\par
  Result (\ref{eq:mainthm01}) does not contradict the upper bound (\ref{eq:zyokai02}) obtained in previous research. Indeed,
  \[
  \forall m=1,2,\ldots,~~\frac{d_1-r+rm}{2m}\leq\frac{d_1}{2}
  \]
  holds (equality holds for $m=1$, which also implies $r=d_1$). Thus, this result can also be seen as a refinement of the previous research (\ref{eq:zyokai02}).

  Result (\ref{eq:mainthm01}) indicates that the learning coefficient is characterized by three quantities: $d_1$, the codimension of the manifold of realizable parameters $\Theta_*\in\mathbb{R}^{d_1+d_2}$; $r$, the rank of the Fisher information matrix; and $m$, the smallest number of times the derivative of $f$ with $r$ first derivatives is linearly independent (for the meaning of $d_1$, see the next section).
\end{rem}

\subsection{Proof of Main Theorem~\ref{mainthm}}
\begin{proof}[\textbf{Proof of Main Theorem~\ref{mainthm}}]\leavevmode\par
By fixing $\tau \in U$, the Taylor expansion of the log-likelihood ratio function $f$ at $\theta=0$ has no constant term according to Lemma~\ref{lem:zitugen}, and is expressed as:
\[
f(x|\theta,\tau) = F_1(x|\theta_1,\ldots,\theta_{r}) + F_m(x|\theta_{r+1},\ldots,\theta_{d_1}) + \text{(higher order terms)}
\]
Given that the parameters $(\theta_1,\ldots,\theta_{d_1})$ satisfy Assumption~\ref{mainass} similar to \cite[Main Theorem 1]{kurumadani1}, equation (\ref{eq_mainthm01}) holds.
\end{proof}

\subsection{Proof of Main Theorem~\ref{mainthm1}}
Let's consider $U(\subset\mathbb{R}^{d_2})$ as a sufficiently small neighborhood around $\tau=0$.

\begin{lem}[Characterization of Linear Independence of Random Variables]\label{lem:linind}\leavevmode\par
Assume $n$ random variables $X_1(\tau),\ldots,X_n(\tau)$ are functions of $(x,\tau) \in \chi \times U$, each being an analytic function of $\tau$ for fixed $x \in \chi$, and define $\Sigma(\tau):=\left(\mathbb{E}\left[X_i(\tau)X_j(\tau)\right]\right)_{1\leq i,j\leq n}$.
  \begin{itemize}
    \item [(1)] For any $\tau \in U$, $\Sigma(\tau)$ is non-negative definite, and for each fixed $\tau \in U$, the following two conditions are equivalent:
      \begin{itemize}
        \item [(a)] $\Sigma(\tau_0)$: positive definite
        \item [(b)] $X_1(\tau_0),\ldots,X_n(\tau_0)$ are linearly independent over $\mathbb{R}$ as random variables
      \end{itemize}
    \item [(2)] If $X_1(0),\ldots,X_n(0)$ are linearly independent over $\mathbb{R}$ at $\tau=0$, then for any fixed $\tau_0 \in U$, $X_1(\tau_0),\ldots,X_n(\tau_0)$ are linearly independent over $\mathbb{R}$, satisfying conditions (a) and (b) from (1).
  \end{itemize}
\end{lem}

\begin{proof}\leavevmode\par
If we define $A(\tau):=(X_1(\tau),\ldots,X_n(\tau))^{\top}$, then $\Sigma(\tau)=\mathbb{E}_X\left[A(\tau)A(\tau)^{\top}\right]$. For any vector $\bm{u}(\tau_0):=\left(u_1(\tau_0), \ldots ,u_n(\tau_0)\right)^{\top}$ in $\mathbb{R}^n$,
  \[
    \bm{u}(\tau_0)^{\top}\Sigma(\tau_0)\bm{u}(\tau_0)
    =\bm{u}(\tau_0)^{\top}\mathbb{E}_X\left[A(\tau_0)A(\tau_0)^{\top}\right]\bm{u}(\tau_0)
    =\mathbb{E}_X\left[\left\|A(\tau_0)^{\top}\bm{u}(\tau_0)\right\|^2\right]\geq 0,
  \]
  and
  \[
    \bm{u}(\tau_0)^{\top}\Sigma(\tau_0)\bm{u}(\tau_0)=0 \Leftrightarrow A(\tau_0)^{\top}\bm{u}(\tau_0) =0~~ \text{a.s.}
  \]
  Hence, conditions (a) and (b) are equivalent. If $D(\tau):=\text{det}\left(\Sigma(\tau)\right)$ and $D(0)>0$, then $\forall \tau_0\in U, D(\tau_0)>0$ follows, fulfilling condition (a) from (1).
\end{proof}

\begin{lem}\label{mainlem:teisu}\leavevmode\par
  \begin{itemize}
    \item [(1)]
      When Assumption~\ref{mainass}(1) is satisfied, for any $\tau \in U$, the following holds:
      \begin{align*}
        \mathbb{E}_X\left[F^2_1(X|\theta_{1},\ldots,\theta_{r})\right]=0 &\ \Leftrightarrow\ (\theta_{1},\ldots,\theta_{r})=0
      \end{align*} 
    \item [(2)]
      Let $a$ be a non-zero constant, and suppose Assumption~\ref{mainass}(1)(2)(3) is satisfied.
      \begin{itemize}
        \item [(i)]
          When $(\theta_{r+1},\ldots,\theta_{d_1})$ satisfies Assumption~\ref{mainass}(3)(i), the following equivalence holds at $\tau=0$:
          \begin{align*}
            \begin{split}
              &\left.\mathbb{E}_X\left[
              \left\{F_1(X|\theta_{1},\ldots,\theta_{r})+aF_m(X|\theta_{r+1},\ldots,\theta_{d_1})\right\}^2
              \right]\right|_{\tau=0}=0 \\
              \Leftrightarrow\ & 
              (\theta_{1},\ldots,\theta_{r})=0,~ 
              F_m(X|\theta_{r+1},\ldots,\theta_{d_1})|_{\tau=0}=0~~ \text{a.s.}
            \end{split}
          \end{align*}
        \item [(ii)]
          If $(\theta_{r+1},\ldots,\theta_{d_1}) \neq 0$ satisfies Assumption~\ref{mainass}(3)(ii), then for any $\tau \in U$:
          \[
          \mathbb{E}_X\left[
              \left\{F_1(X|\theta_{1},\ldots,\theta_{r})+aF_m(X|\theta_{r+1},\ldots,\theta_{d_1})\right\}^2
              \right]
          >0
          \]
      \end{itemize}
      In particular, in either case (i) or (ii), if $(\theta_1,\ldots,\theta_r)\neq0$, then for any $\tau \in U$:
      \[
      \mathbb{E}_X\left[
      \left\{F_1(X|\theta_{1},\ldots,\theta_{r})+aF_m(X|\theta_{r+1},\ldots,\theta_{d_1})\right\}^2
      \right]
      >0
      \]
  \end{itemize}
\end{lem}

\begin{proof}\leavevmode\par

\begin{itemize}
  \item [(1)]
    Given that the $r$ random variables, functions of $\tau$,
    \begin{equation}\label{eq:mainlemteisu01}
    \left.\frac{\partial f(X|\theta,\tau)}{\partial\theta_{1}}\right|_{(\theta_{1},\ldots,\theta_{d_1})=0}
    ,\ldots,
    \left.\frac{\partial f(X|\theta,\tau)}{\partial\theta_{r}}\right|_{(\theta_{1},\ldots,\theta_{d_1})=0}
    \end{equation}
    are linearly independent at $\tau=0$ according to Assumption~\ref{mainass}(1), and by Lemma~\ref{lem:linind}(2), we fix any $\tau_0 \in U$ to conclude that (\ref{eq:mainlemteisu01}) is linearly independent over $\mathbb{R}$.
    Thus, we obtain the following equivalence:
    \begin{align*}
      \mathbb{E}_X\left[F^2_1(X|\theta_{1},\ldots,\theta_{r})\right]=0 & \Leftrightarrow ~F_1(X|\theta_1,\ldots,\theta_r)=0~~\text{a.s.}\\
      & \Leftrightarrow ~\sum_{k=1}^r \theta_k
      \left.\frac{\partial f(X|\theta,\tau_0)}{\partial\theta_{k}}\right|_{(\theta_{1},\ldots,\theta_{d_1})=0}=0~~\text{a.s.}\\
      & \Leftrightarrow ~(\theta_{1},\ldots,\theta_{r})=0
    \end{align*} 

  \item[(2)]
    When Assumption~\ref{mainass}(3)$(i)$ is met, $F_m(X|\theta_{r+1},\ldots,\theta_{d_1})|_{\tau=0}=0$ (a.s.) holds, and thus using (1), $(i)$ is demonstrated.
    Therefore, it only remains to demonstrate when Assumption~\ref{mainass}(3)$(ii)$ is met. Given that $a\cdot F_m(X|\theta_{r+1},\ldots,\theta_{d_1})$ and (\ref{eq:mainlemteisu01}) are linearly independent at $\tau=0$,
    by Lemma~\ref{lem:linind}(2), for any fixed $\tau_0 \in U$, $a\cdot F_m(X|\theta_{r+1},\ldots,\theta_{d_1})$ and (\ref{eq:mainlemteisu01}) remain linearly independent over $\mathbb{R}$. Thus, when $\tau = \tau_0$, we have
    \begin{align*}
    &\mathbb{E}_X\left[
    \left\{F_1(X|\theta_{1},\ldots,\theta_{r})+aF_m(X|\theta_{r+1},\ldots,\theta_{d_1})\right\}^2
    \right]
    =0\\
     \Leftrightarrow & ~\sum_{k=1}^r \theta_k
      \left.\frac{\partial f(X|\theta,\tau_0)}{\partial\theta_{k}}\right|_{(\theta_{1},\ldots,\theta_{d_1})=0}
      +a\cdot F_m(X|\theta_{r+1},\ldots,\theta_{d_1})=0~~\text{a.s.}
    \end{align*}
      This contradicts the linear independence over $\mathbb{R}$ of $a\cdot F_m(X|\theta_{r+1},\ldots,\theta_{d_1})$ and (\ref{eq:mainlemteisu01}). Therefore, this equivalence does not hold for any $\tau \in U$, and
      \[
      \mathbb{E}_X\left[
    \left\{F_1(X|\theta_{1},\ldots,\theta_{r})+aF_m(X|\theta_{r+1},\ldots,\theta_{d_1})\right\}^2
    \right]>0
      \]
      is obtained.
\end{itemize}
\end{proof}

\begin{proof}[\textbf{Proof of Main Theorem~\ref{mainthm1}}]\leavevmode\par
For any $\tau_0 \in U$, in the Taylor expansion of $K(\theta,\tau_0)$ at $\theta=0$, Lemma~\ref{lem:zitugen} confirms that there is no constant term.

Furthermore, since the parameters $(\theta_1,\ldots,\theta_{d_1})$ meet Assumption~\ref{mainass} similar to \cite[Main Theorem 2]{kurumadani1}, applying Lemma~\ref{mainlem:teisu} in place of \cite[Lemma 2.2]{kurumadani1} in the proof of \cite[Main Theorem 2]{kurumadani1} demonstrates the theorem.
\end{proof}

\newpage
\section{Variable Transformations to Satisfy Assumption~\ref{mainass}}

\subsection{Assumption~\ref{mainass}(0)}
The example treated in Section~\ref{ch_ex} inherently satisfied Assumption~\ref{mainass}(0), but this is not generally guaranteed. However, if the realizable parameter $\theta_* \in \Theta_* \subset \mathbb{R}^{d_1+d_2}$ is a nonsingular point of the $d_2$-dimensional analytic manifold $\Theta_*$, it naturally satisfies Assumption~\ref{mainass}(0).

\begin{dfn}\cite[Definition 2.6]{watanabe1}\leavevmode\par
A point $\theta_* \in \Theta_*$ is considered nonsingular if there exists open sets $U, V \subset \mathbb{R}^{d_1+d_2}$ and an analytic isomorphism $\varphi: U \rightarrow V$ such that:
\[
\varphi(\Theta_* \cap U) = \{(0,\ldots,0,x_1,\ldots,x_{d_2})|x_i\in\mathbb{R}\} \cup V
\]  
\end{dfn}

\begin{rem}\leavevmode\par
A concrete method to construct $\varphi$ involves solving the $d_1$ defining equations of $\Theta_*$ at the point $\theta_*$ with respect to $\theta_1,\ldots,\theta_{d_1}$, a direct consequence of the Implicit Function Theorem \cite[Remark 2.2]{watanabe1}.
\end{rem}

\begin{ex}\leavevmode\par
Consider a statistical model with parameters $(a,b)$:
\[
p(x|a,b) = \frac{1}{1 + e^{-(a+1)(b+1)x}},\ \  q(x) = \frac{1}{1 + e^{-x}}~~ (x=1,-1)
\]
Clearly,
\[
\Theta_* = \left\{
(a,b) ~ | ~ (a+1)(b+1) = 1
\right\} \subset \mathbb{R}^2
\]
constitutes a nonsingular one-dimensional manifold. For example, at the origin, solving the defining equation of $\Theta_*$ for $a$ yields:
\[
a = -\frac{b}{b+1}
\]
Thus, the coordinate transformation $(a,b) \mapsto (\tilde{a}, b)$ that satisfies Assumption~\ref{mainass}(0) is:
\[
\tilde{a} := a + \frac{b}{b+1}
\]
Indeed, in the neighborhood of the origin $(\tilde{a}, b) = (0, 0)$, we find
\[
\Theta_* = \left\{(\tilde{a}, b)~|~\tilde{a} = 0, \forall b\right\}
\]
is satisfied, and $\Theta_*$ is defined by the single parameter $\tilde{a}$ near the origin.

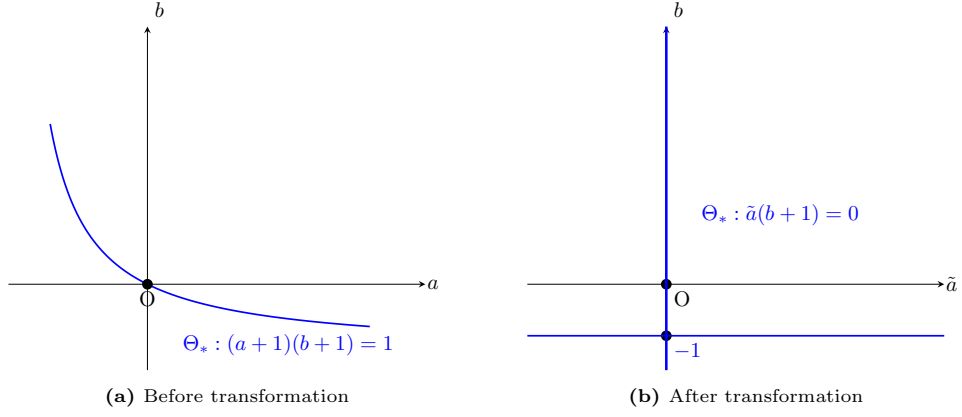
\begin{figure}[ht]
\begin{minipage}[b]{0.48\columnwidth}
  \centering

\scalebox{0.8}[0.8]{
  \begin{tikzpicture}
    \begin{axis}[
        axis lines=center,
        ticks=none,
        xmin=-5, xmax=10, ymin=-5, ymax=15,
        xlabel=$a$, ylabel=$b$, 
        every axis x label/.style={
          at={(ticklabel* cs:1.05)},
          anchor=east,
        },
        every axis y label/.style={
          at={(ticklabel* cs:1.05)},
          anchor=west,
        }
      ]
      \addplot [blue,thick,samples=200,domain=-3.5:8] {20/(x+5)-4 };
      \draw(axis cs:0,0)node[below]{O};
      \draw[blue](axis cs:1,-3.5)node[right]{$\Theta_*: (a+1)(b+1)=1$};
      \fill (axis cs:0,0) circle (2.5pt);
    \end{axis}
  \end{tikzpicture}
}
  \subcaption{Before transformation}
\end{minipage}
\hspace{0.03\columnwidth}
\begin{minipage}[b]{0.48\columnwidth}
  \centering

\scalebox{0.8}[0.8]{
  \begin{tikzpicture}
    \begin{axis}[
        axis lines=center,
        ticks=none,
        xmin=-5, xmax=10, ymin=-5, ymax=15,
        xlabel=$\tilde{a}$, ylabel=$b$, 
        every axis x label/.style={
          at={(ticklabel* cs:1.05)},
          anchor=east,
        },
        every axis y label/.style={
          at={(ticklabel* cs:1.05)},
          anchor=west,
        }
      ]
      \draw(axis cs:0,0)node[below right]{O};
      \draw[blue](axis cs:0,-3)node[below right]{$-1$};
      \draw[blue](axis cs:1,3)node[above right]{$\Theta_*:\tilde{a}(b+1)=0$};
      \fill (axis cs:0,0) circle (2.5pt);
      \fill (axis cs:0,-3) circle (2.5pt);
      \draw [blue,very thick] (axis cs:0,-5) -- (axis cs:0,15);
      \draw [blue,thick] (axis cs:-5,-3) -- (axis cs:10,-3);
    \end{axis}
  \end{tikzpicture}
}
\subcaption{After transformation}
\end{minipage}
\caption{Changes in the set of realizable parameters $\Theta_*$ before and after coordinate transformation}\label{fig:henkan}
\end{figure}

\end{ex}

\subsection{Assumption~\ref{mainass}(1)-(3)}
To satisfy Assumption~\ref{mainass}(1)-(3), variable transformations are derived from linear dependencies, similar to those described in \cite{kurumadani1}. The $r$ parameters that satisfy Assumption~\ref{mainass}(1) are those that form a basis of the $\mathbb{R}$ vector space composed of the first derivatives of $f$ with $\tau=0$, akin to the approach in \cite{kurumadani1}.

For Assumption~\ref{mainass}(2), attention is needed in this paper's context as follows:

In the following, we denote by $U(\subset\mathbb{R}^{d_2})$ a sufficiently small neighborhood of $\tau=0$, and we define $L$ as the set (commutative ring) of all analytic functions on $U$, i.e.,
\[
L := \{\varphi(\tau) : U \rightarrow \mathbb{R} : \text{analytic}\}.
\]

Generally, for elements $v_1,\ldots,v_n$ in a module $M$ over a commutative ring $L$, being linearly independent over $L$ means that:
\[
l_1,\ldots,l_n \in L,~~\sum_{i=1}^n l_iv_i=0 \Rightarrow l_1=\cdots=l_n=0 \in L
\]

If Assumption~\ref{mainass}(1) is satisfied, it also implies linear independence over the commutative ring $L$.

To satisfy Assumption~\ref{mainass}(2), it is necessary to construct coordinate transformations from linear dependencies over the ring $L$. For example, using $\varphi_1(\tau),\ldots,\varphi_r(\tau) \in L$, if we can express:
\begin{equation}\label{eq_mainrem01}
\left.
\frac{\partial^{i_{r+1}+\cdots +i_{d_1}} f}{\partial\theta_{r+1}^{i_{r+1}}\cdots\partial\theta_{d_1}^{i_{d_1}}}
\right|_{(\theta_1,\ldots,\theta_{d_1})=0}
=\sum_{k=1}^r\varphi_k(\tau)
  \left.\frac{\partial f}{\partial\theta_k}\right|_{(\theta_1,\ldots,\theta_{d_1})=0}
  \text{ a.s.}
\end{equation}
Then, the coordinate transformation
\[
\theta_k^{\prime}=\theta_k+\frac{\varphi_k(\tau)}{i_{r+1}!\cdots i_{d_1}!}\theta_{r+1}^{i_{r+1}}\cdots\theta_{d_1}^{i_{d_1}}
\]
for $k=1,\ldots,r$ ensures that after the transformation:
\begin{equation*}
\left.
\frac{\partial^{i_{r+1}+\cdots +i_{d_1}} f}{\partial\theta_{r+1}^{i_{r+1}}\cdots\partial\theta_{d_1}^{i_{d_1}}}
\right|_{(\theta_1^{\prime},\ldots,\theta_r^{\prime},\theta_{r+1},\ldots\theta_{d_1})=0}
=0 \text{ a.s.}
\end{equation*}
Thus, one should repeatedly apply such coordinate transformations each time a linear relation (\ref{eq_mainrem01}) is identified from the lower order terms.

\begin{rem}\leavevmode\par
If instead of the relation in (\ref{eq_mainrem01}), we have a relation using $\varphi \in L, \varphi(0)=0$:
\[
\varphi(\tau)\left.
\frac{\partial^{i_{r+1}+\cdots +i_{d_1}} f}{\partial\theta_{r+1}^{i_{r+1}}\cdots\partial\theta_{d_1}^{i_{d_1}}}
\right|_{(\theta_1,\ldots,\theta_{d_1})=0}
=\sum_{k=1}^r\varphi_k(\tau)
  \left.\frac{\partial f}{\partial\theta_k}\right|_{(\theta_1,\ldots,\theta_{d_1})=0}
  \text{ a.s.}
\]
This form cannot be transformed into the format of (\ref{eq_mainrem01}) because $\varphi^{-1} \notin L$. Such a situation, where $(\theta_1,\theta_2,\tau)=0$ becomes a singular point in the set of realizable parameters and necessitates further blow-ups, does not occur in the discussions in \cite{kurumadani1}, which are based on the vector space over the field $\mathbb{R}$.

A situation like this may occur, for example, when considering a random variable $X$ that follows a binomial distribution Bin$(2,\theta)$ with parameter $\theta (0 < \theta < 1)$:
  \[
  \tilde{p}(X=x|\theta)=\binom{2}{x}\theta^x(1-\theta)^{2-x}=
  \begin{cases}
    (1-\theta)^2, & (x=0)\\
    2\theta(1-\theta), & (x=1)\\
    \theta^2, & (x=2)
  \end{cases}
  \]

Consider the following mixed distribution model with parameters $(\theta_1, \theta_2, \tau)$:
  \begin{align*}
  p(X=x|\theta_1, \theta_2, \tau):=
  \tau\cdot\tilde{p}\left(X=x\middle|\theta_2+\frac{1}{2}\right)
  +(1-\tau)\cdot\tilde{p}\left(X=x\middle|\theta_1+\frac{1}{2}\right)&\\
  (x=0,1,2)\notag&
  \end{align*}

If the true distribution is $\tilde{p}(X|1/2)$, this model realizes the true distribution at the point $(\theta_1, \theta_2, \tau)=0$. Consider the derivatives of $f$ at this point:
  \[
  p(\cdot|\theta_1=0, \theta_2=0, \forall\tau)
  =\tau\cdot\tilde{p}\left(\cdot\middle|\frac{1}{2}\right)+(1-\tau)\cdot\tilde{p}\left(\cdot\middle|\frac{1}{2}\right)
  =q(\cdot)
  \]
shows that this statistical model realizes the true distribution only at $\theta=0$. Therefore, treat $\tau$ as an arbitrary constant in the neighborhood $U$ of 0:
  \begin{gather*}
    \left.\frac{\partial f}{\partial\theta_1}\right|_{(\theta_1, \theta_2)=0} = 4(1-\tau)(1-X), ~~ 
    \left.\frac{\partial f}{\partial\theta_2}\right|_{(\theta_1, \theta_2)=0} = 4\tau(1-X)
  \end{gather*}  
  These are linearly dependent as:
  \[
  \left.\frac{\partial f}{\partial\theta_2}\right|_{(\theta_1, \theta_2)=0} = 
  \frac{\tau}{1-\tau}\cdot\left.\frac{\partial f}{\partial\theta_1}\right|_{(\theta_1, \theta_2)=0}
  \]
  but
  \[
  \left.\frac{\partial f}{\partial\theta_1}\right|_{(\theta_1, \theta_2)=0} = 
  \frac{1-\tau}{\tau}\cdot\left.\frac{\partial f}{\partial\theta_2}\right|_{(\theta_1, \theta_2)=0}
  \]
  is inappropriate because the function $(1-\tau)/\tau$ is not defined near $\tau=0$.

Additionally, in this example, performing the coordinate transformation to eliminate the first derivative with respect to $\theta_2$,
  \[
  \theta_1^{\prime}:=\theta_1+\frac{\tau}{1-\tau}\theta_2
  \]
  results in
  \[
  \left.
  \frac{\partial^2 f}{\partial\theta_2^2}
  \right|_{(\theta_1^{\prime}, \theta_2)=0}
  =-\frac{8\tau}{1-\tau}(1-4X+2X^2)
  \]
  and becomes zero at $\tau=0$, satisfying Assumption~\ref{mainass}(3)(i) but not (ii). Therefore, while blow-ups can proceed under Main Theorem~\ref{mainthm1}, applying Main Theorem~\ref{mainthm1} directly to compute the real log canonical threshold at the origin is not possible.

This example illustrates that just applying Main Theorem~\ref{mainthm1} once is insufficient for calculating the real log canonical threshold at the origin $(\theta_1, \theta_2, \tau)=0$, which is a singular point in the set of realization parameters, necessitating further blow-ups.
\end{rem}

\newpage
\section{Example}
In some cases, using Main Theorem~\ref{mainthm1}, we can calculate the real log canonical thresholds at non-singular points in the set of realizable parameters, $\Theta_*$. 

\subsection{Mixed Distribution Model}

\begin{ex}\label{ex_kongou02}\leavevmode\par
  Let $M(\geq 2)$ be a constant, and $\tilde{p}(x|\theta)$ represents a binomial distribution Bin$(M,\theta)$, i.e.,
  \[
  \tilde{p}(X=x|\theta)=\binom{M}{x}\theta^x(1-\theta)^{M-x}~~(x=0,1,\ldots,M)
  \]

  Consider a mixture of $H(\geq 2)$ such probability distributions $\tilde{p}(x|\theta)$:
  \begin{equation}\label{eq:kongou00}
  p(x |\theta_1,\ldots,\theta_{H},\tau_1,\ldots,\tau_{H-1}):=
  \sum_{i=1}^{H-1}\tau_i\cdot\tilde{p}\left(x\middle|\theta_{i}\right) 
  + \left(1-\sum_{i=1}^{H-1}\tau_i\right)\cdot\tilde{p}\left(x\middle|\theta_H\right)
  \end{equation}
  Here, parameters $\tau_i$ represent mixing proportions $(0\leq\tau_i\leq1,\ \sum_{i=1}^{H-1}\tau_i\leq 1)$, and the parameter space is $\Theta=[0,1]^H\times[0,1]^{H-1}$.

  When the number of mixture components in the true distribution $q(x)$, $H_0$, satisfies $H_0 \leq \text{min}\{H, M/2\}$, the upper bound for the learning coefficient of the mixed distribution model (\ref{eq:kongou00}) is given by
  \[
  \frac{3H_0+H-2}{4}
  \]
  (multiplicity is 1).
\end{ex}

\begin{proof}\leavevmode\par
The true distribution $q(x)$ can be represented using $H_0-1$ parameters $\{\tau_{i*}\}$ and $H_0$ mutually distinct parameters $\{\theta_{i*}\}$:
\begin{align*}
  q(x)=\sum_{i=1}^{H_0-1}\tau_{i*}\cdot\tilde{p}(x|\theta_{i*}) 
    + \left(1-\sum_{i=1}^{H_0-1}\tau_{i*}\right)\cdot\tilde{p}(x|\theta_{H_0*})
\end{align*}
where $\tau_{i*}>0, \sum_{i=1}^{H_0-1}\tau_{i*}<1$. For $H_0=1$, represent it as:
\[
  q(x)=\tilde{p}(x|\theta_{*})
\]
Note that the statistical model (\ref{eq:kongou00}) realizes $q(x)$ at the point:
\begin{align*}
  (\theta_1,\ldots,\theta_{H_0-1},\theta_{H_0},\ldots,\theta_{H-1},\theta_{H})
    &= (\theta_{1*},\ldots,\theta_{H_0-1*},\theta_{H_0*},\ldots,\theta_{H_0*},\theta_{H_0*})\\
  (\tau_1,\ldots,\tau_{H_0-1},\tau_{H_0},\ldots,\tau_{H-1})
    &= (\tau_{1*},\ldots,\tau_{H_0-1*},\forall\tau_{H_0},\ldots,\forall\tau_{H-1})
\end{align*}
After translating to:
\[
p(x |\theta,\tau)=\sum_{i=1}^{H_0-1}(\tau_i+\tau_{i*})\tilde{p}(x|\theta_i+\theta_{i*})
+\sum_{i=H_0}^{H-1}\tau_i\tilde{p}(x|\theta_i+\theta_{H_0*})
+(1-\sum_{i=1}^{H_0-1}\tau_{i*}-\sum_{i=1}^{H-1}\tau_i)\tilde{p}(x|\theta_H+\theta_{H_0*})
\]
and considering any point satisfying:
\begin{equation}\label{eq:teisu01}
(\theta_1,\ldots,\theta_{H_0-1},\theta_{H_0},\ldots,\theta_{H-1},\theta_{H},\tau_1,\ldots,\tau_{H_0-1})=0    
\end{equation}
where the remaining parameters $(\tau_{H_0},\ldots,\tau_{H-1})$ are treated as constants.
Here, we consider when $(\tau_{H_0}, \ldots, \tau_{H-1})$ satisfy the following condition:
\begin{equation}\label{eq:teisu04}
\tau_{H_0}, \ldots, \tau_{H-1} > 0,\ \ \sum_{k=1}^{H_0-1}\tau_{k*} + \sum_{k=H_0}^{H-1}\tau_k < 1
\end{equation}
The first derivative of the log-likelihood ratio function $f$ at point (\ref{eq:teisu01}) is:
\begin{align*}
   \frac{\partial f}{\partial\theta_i} &= -\frac{1}{q(x)}
   \left\{\begin{array}{ll}
     \tau_{i*}\frac{\partial \tilde{p}}{\partial\theta}(x|\theta_{i*}) & i=1,\ldots,H_0-1\\
     \tau_{i}\frac{\partial \tilde{p}}{\partial\theta}(x|\theta_{H_0*}) & i=H_0,\ldots,H-1\\
     (1-\sum_{k=1}^{H_0-1}\tau_{k*}-\sum_{k=H_0}^{H-1}\tau_{k})\frac{\partial \tilde{p}}{\partial\theta}(x|\theta_{H_0*}) & i=H
   \end{array}\right.\\
   \frac{\partial f}{\partial\tau_i} &= -\frac{\tilde{p}(x|\theta_{i*}) - \tilde{p}(x|\theta_{H_0*})}{q(x)}, \ \ \ \ \ \ \ \ \ \ \ \ \ \ \ \ \ \ \ \ \ \ \ \ \ \ \ \ \ \ \ \ \ \ \ \ \ \ \ i=1,\ldots,H_0-1
\end{align*}
Here, to satisfy the linear dependency
\[
\frac{\partial f}{\partial\theta_i} = \frac{\tau_i}{1-\sum_{k=1}^{H_0-1}\tau_{k*}-\sum_{k=H_0}^{H-1}\tau_{k}}
\cdot \frac{\partial f}{\partial\theta_H}
,~ ~ (i=H_0,\ldots,H-1)
\]
we perform a coordinate transformation
\[
\theta_H^{\prime} := \theta_H + \frac{\sum_{i=H_0}^{H-1}\tau_i\theta_i}{1-\sum_{k=1}^{H_0-1}\tau_{k*}-\sum_{k=H_0}^{H-1}\tau_{k}}
\]
which results in
\[
\frac{\partial f}{\partial\theta_{H_0}} = \cdots = \frac{\partial f}{\partial\theta_{H-1}} = 0
\]
and these second derivatives are
$\boldsymbol{\theta} := (\theta_{H_0} \cdots \theta_{H-1})^{\top} \in \mathbb{R}^{H-H_0}$
given by
\begin{align*}
F_2(x, \boldsymbol{\theta})
=& \sum_{\substack{i_{H_0}+\cdots+i_{H-1}=2\\ \{i_{H_0},\ldots,i_{H-1}\}\subset\mathbb{Z}_{\geq 0}}}
  \frac{1}{i_{H_0}!\cdots i_{H-1}!}\times
  \left.\frac{\partial^2 f}{\partial\theta_{H_0}^{i_{H_0}}\cdots\partial\theta_{H-1}^{i_{H-1}}}
  \right|_{\theta=0}
  \times\theta_{H_0}^{i_{H_0}}\cdots\theta_{H-1}^{i_{H-1}}\\
=& \frac{-1}{1-\sum_{k=1}^{H_0-1}\tau_{k*}-\sum_{k=H_0}^{H-1}\tau_{k}}
  \cdot \frac{1}{2q(x)} \cdot
  \frac{\partial^2\tilde{p}}{\partial\theta^2}(x|\theta_{H_0*})
  \cdot \boldsymbol{\theta}^{\top}\Sigma\boldsymbol{\theta}
\end{align*}
where
\begin{align*}
 \Sigma &:= (\sigma_{i,j})_{H_0\leq i,j\leq H-1}\\
 \sigma_{i,j} &= \left\{\begin{array}{ll}
   \tau_i(1-\sum_{k=1}^{H_0-1}\tau_{k*}-\sum_{k=H_0}^{H-1}\tau_{k}+\tau_i) & i=j\\
   \tau_i\tau_j & i\neq j
   \end{array}\right.
\end{align*}
Under condition (\ref{eq:teisu04}), by \cite[Lemma 4.2]{kurumadani1}, $\Sigma$ is a positive definite matrix, and there exist no non-trivial parameters $\boldsymbol{\theta}$ satisfying $F_2(x, \boldsymbol{\theta}) = 0$.

By Lemma~\ref{lem:kongou01}, the following equation holds:
\begin{equation}\label{eq:kongou01}
\tilde{p}(x|\theta_{1*}),\ldots,\tilde{p}(x|\theta_{H_0*}),\ 
\frac{\partial\tilde{p}}{\partial\theta}(x|\theta_{1*}),\ldots,\frac{\partial\tilde{p}}{\partial\theta}(x|\theta_{H_0*}),\ 
\frac{\partial^2\tilde{p}}{\partial\theta^2}(x|\theta_{H_0*}):~\text{linearly independent}
\end{equation}
so for any parameter $\boldsymbol{\theta} (\neq 0)$,
\[
\frac{\partial f}{\partial\theta_1},\ldots,\frac{\partial f}{\partial\theta_{H_0-1}},
\frac{\partial f}{\partial\theta_{H}},
\frac{\partial f}{\partial\tau_1},\ldots,\frac{\partial f}{\partial\tau_{H_0-1}},
F_2(x, \boldsymbol{\theta})
\]
are linearly independent, and for $(d_1,r,m)=(H+H_0-1,2H_{0}-1,2)$, applying Main Theorem~\ref{mainthm1} yields the real log canonical threshold
\[
\lambda = \frac{d_1-r+rm}{2m} = \frac{3H_{0}+H-2}{4}
\]
(multiplicity is 1).\\
\end{proof}

In the following lemma, for simplicity, the previously used symbol $H_0$ is denoted by $N$, and $\theta_{i*}$ is denoted by $\theta_i$.

\begin{lem}\label{lem:kongou01}\leavevmode\par
Let $M(\geq 2)$ be an integer, $\tilde{p}(x|\theta)$ represent a binomial distribution Bin$(M,\theta)$ parameterized by $\theta$, and $N(\leq M/2)$ be an integer greater than or equal to one. Then, for a set of $N$ distinct parameters $\{\theta_1,\ldots,\theta_N\}\subset[0,1]$, the following random variables are linearly independent:
  \[
  \tilde{p}(x|\theta_{1}),\ldots,\tilde{p}(x|\theta_{N}),\ 
\frac{\partial\tilde{p}}{\partial\theta}(x|\theta_{1}),\ldots,\frac{\partial\tilde{p}}{\partial\theta}(x|\theta_{N}),\ 
\frac{\partial^2\tilde{p}}{\partial\theta^2}(x|\theta_{N})
  \]

\end{lem}

\begin{proof}\leavevmode\par
  Using real numbers $\{s_i\}, \{t_i\}, u$, assume that:
  \begin{equation}\label{eq:kongou02}
  \sum_{i=1}^N s_i\tilde{p}(x|\theta_i)
  +\sum_{i=1}^N t_i\frac{\partial\tilde{p}}{\partial\theta}(x|\theta_i)
  +u\frac{\partial^2\tilde{p}}{\partial\theta^2}(x|\theta_N)
  =0
  \end{equation}
  We need to show $s_i=t_i=u=0~(i=1,\ldots,N)$. First, summing (\ref{eq:kongou02}) over $x=0,1,\ldots,M$ gives:
  \begin{equation}\label{eq:kongou03}
    \sum_{i=1}^Ns_i=0
  \end{equation}
  Next, multiplying (\ref{eq:kongou02}) by $x$ and summing over $x=0,1,\ldots,M$, and using (\ref{eq:kongou03}), we get:
  \begin{equation}\label{eq:kongou04}
    \sum_{i=1}^Ns_i\theta_i+\sum_{i=1}^N t_i=0
  \end{equation}  
  Further, multiplying (\ref{eq:kongou02}) by $x^2$ and summing over $x=0,1,\ldots,M$, and manipulating (\ref{eq:kongou04}), results in:
  \begin{equation}\label{eq:kongou05}
    \sum_{i=1}^Ns_i\theta_i^2+2\sum_{i=1}^N t_i\theta_i+2u=0
  \end{equation}  
  Notably, for a binomial distribution $X\sim$Bin$(M,\theta)$, the $n(\leq M)$th central moment is a polynomial in $\theta$ of degree $n$. Repeating the same argument for $n=2N(\leq M)$, we obtain:
  \[
  \sum_{i=1}^Ns_i\theta_i^{2N}
  +{}_{2N} \mathrm{P}_{1}\sum_{i=1}^{N} t_i\theta_i^{2N-1}
  +{}_{2N} \mathrm{P}_{2}u\theta_N^{2N-2}
  =0
  \]
  Here, ${}_m \mathrm{P}_n:=\frac{m!}{(m-n)!}$ for non-negative integers $m\geq n$. we define the matrix $W$ as follows:
\[
W:=
\left(
\begin{array}{ccc|ccc|c}
  1 & \cdots & 1 & 0 & \cdots & 0 & 0   \\
  {}_{1} \mathrm{P}_{0}\theta_1 & \cdots & {}_{1} \mathrm{P}_{0}\theta_N & 1 & \cdots & 1 & 0  \\
  {}_{2} \mathrm{P}_{0}\theta^2_1 & \cdots & {}_{2} \mathrm{P}_{0}\theta_N^2 & {}_{2} \mathrm{P}_{1}\theta_1 & \cdots & {}_{2} \mathrm{P}_{1}\theta_N & {}_{2} \mathrm{P}_{2}  \\
  {}_{3} \mathrm{P}_{0}\theta^3_1 & \cdots & {}_{3} \mathrm{P}_{0}\theta^3_N & {}_{3} \mathrm{P}_{1}\theta_1^2 & \cdots & {}_{3} \mathrm{P}_{1}\theta_N^2 & {}_{3} \mathrm{P}_{2}\theta_N  \\
  {}_{4} \mathrm{P}_{0}\theta^4_1 & \cdots & {}_{4} \mathrm{P}_{0}\theta^4_N & {}_{4} \mathrm{P}_{1}\theta^3_1 & \cdots & {}_{4} \mathrm{P}_{1}\theta_N^3 & {}_{4} \mathrm{P}_{2}\theta^2_N  \\
  \vdots &  \cdots & \vdots & \vdots & \cdots & \vdots & \vdots \\
  {}_{2N} \mathrm{P}_{0}\theta^{2N}_1 & \cdots & {}_{2N} \mathrm{P}_{0}\theta^{2N}_N & {}_{2N} \mathrm{P}_{1}\theta^{2N-1}_1 &\cdots &  {}_{2N} \mathrm{P}_{1}\theta^{2N-1}_{N} & {}_{2N} \mathrm{P}_{2}\theta^{2N-2}_N   \\
\end{array}
\right)
\]
Using this, $\{s_i\}, \{t_i\}, u$ are the solutions to:
\[
W(s_1~\cdots~s_N~t_1~\cdots~t_N~u)=\boldsymbol{0}~\in\mathbb{R}^{2N+1}.
\]
Therefore, it is sufficient to show that $W$ is an invertible matrix. A straightforward calculation shows:
\[
\text{det }W\propto \prod_{1\leq i<j\leq N-1}(\theta_i-\theta_j)^4\times\prod_{1\leq i\leq N-1}(\theta_i-\theta_N)^6,
\]
and since $\{\theta_1,\ldots,\theta_N\}$ are distinct, $\text{det }W\neq0$, meaning $W$ is indeed an invertible matrix.

\end{proof}

\begin{rem}\leavevmode\par
If we restrict the parameter space $\Theta$ such that all realization parameters satisfy condition (\ref{eq:teisu04}), then this upper bound becomes the learning coefficient itself.

Furthermore, in the proof, the fact that the distribution is binomial is not essential; what is essential is the linear independence as specified in (\ref{eq:kongou01}). For instance, it can be verified that this condition is also satisfied for a Poisson distribution with mean $\theta$. Regarding mixed Poisson distributions, the learning coefficients are known from \cite{sato1}, and the upper bound obtained in Example~\ref{ex_kongou02} coincides with the learning coefficients given in \cite{sato1}.
\end{rem}

\subsection{Reduced Rank Regression Model}
Consider a three-layer neural network with $M, H, N$ units in the input, hidden, and output layers, respectively. Define the statistical model with input $x \in \mathbb{R}^M$ and output $y \in \mathbb{R}^N$ by
  \begin{align*}
  p(y|x,\theta):=
  \frac{1}{(2\pi)^{N/2}}
  \exp\left(-\frac{1}{2}||y-BAx||^2\right),
  \end{align*}
where $||~ ||$ denotes the norm in the $N$-dimensional Euclidean space. Let the parameter space be
  \[
  \Theta:=\left\{\theta=(A,B)| A\in \mathbb{R}^{H\times M},\ B\in \mathbb{R}^{N\times H}\right\}.
  \]
Such a model is referred to as a reduced rank regression model.

Let the true distribution parameters be $(A_*, B_*)$, and the rank of the matrix $B_*A_*$ be $r$. The learning rate for this model has been established for all combinations $(M,H,N,r)$ \cite{Aoyagi1}. Here, we calculate the learning rate for specific values of $(M,H,N,r)$ and verify that it coincides with previous studies.

\begin{rem}
The model is semi-regular when $r>0$, but not when $r=0$.
\end{rem}

We assume that the input space $x$ is non-degenerate, i.e., $x = (x_1, \ldots, x_M)^\top$ is assumed to be stochastically independent.

Let $z:=y-B_*A_*x \in \mathbb{R}^N$. By assumption, given $x$, $z|x$ follows the $N$-dimensional standard normal distribution $N(0, I_N)$. Also, the log-likelihood ratio function $f$ can be represented as
  \begin{align*}
    f(x,y|\theta) &= \log \frac{p(x,y|\theta_*)}{p(x,y|\theta)} \\
    &= \frac{1}{2}\left\{||y-BAx||^2 - ||y-B_*A_*x||^2\right\} \\
    &= \frac{1}{2}\left\{||z-Sx||^2 - ||z||^2\right\} \\
    &= \frac{1}{2}||Sx||^2 - z^\top Sx,
  \end{align*}
where $S := BA - B_*A_*$.

From the linear independence of $x$, the following lemma follows.

  \begin{lem}\label{lem:rrr01}\leavevmode\par
    \begin{itemize}
      \item [(1)] 
        The $MN$ random variables $z_ix_j\ (i=1,\ldots,N, j=1,\ldots,M)$ are linearly independent.
      \item [(2)]
        The set of realized parameters is given by
        \[
        \Theta_*=\left\{(A,B)|S=0\right\}=\left\{(A,B)|BA=B_*A_*\right\}.
        \]
    \end{itemize}
  \end{lem}

\begin{proof}\leavevmode\par
  \begin{itemize}
    \item [(1)]
      For $T \in \mathbb{R}^{N \times M}$, show that
      \[
      z^\top T x = 0 ~ \text{a.s.} \Rightarrow T = 0.
      \]
      Let $W := z^\top T x$. Due to the reproductive property of the normal distribution, $W|x \sim N(0, ||Tx||^2)$. Thus, by assumption, $||Tx||^2 = 0$, i.e., $Tx = 0$ (a.s.), and from the independence of $x$, it follows that $T = 0$.
    \item [(2)]
      For the expected value of
      \[
      f(x,y|\theta) = \frac{1}{2}||Sx||^2 - z^\top Sx,
      \]
      note that $\mathbb{E}[z^\top Sx] = \mathbb{E}[\mathbb{E}[z^\top Sx | x]] = 0$, hence
      \[
      K(\theta) = \frac{1}{2}\mathbb{E}\left[||Sx||^2\right].
      \]
      In the set of realization parameters, $K(\theta) = 0$, meaning $||Sx||^2 = 0$, i.e., $Sx = 0$ (a.s.), and from the independence of $x$, $S = 0$ follows.
  \end{itemize}  
\end{proof}

\begin{ex}\leavevmode\par
  For $(H, M, N, r) = (2, 1, 2, 1)$, any realization parameters of the statistical model $p$ satisfy Assumption~\ref{mainass} when $(d_1, r, m) = (2, 2, 1)$. Notably, the learning coefficient is 1 (multiplicity is 1).
\end{ex}

\begin{rem}
  This result is consistent with previous studies \cite{Aoyagi1}.
\end{rem}

\begin{proof}\leavevmode\par
  Since the rank of $B_*A_* \in \mathbb{R}^{N\times M} (=\mathbb{R}^{2\times 1})$ is 1, there exist invertible matrices $P, Q$ such that
  \[
  P^{-1}B_*A_*Q^{-1}=\left(\begin{array}{c}1 \\ 0 \end{array}\right).
  \]
  By substituting parameters $(A, B)$ with $B' := P^{-1}B$ and $A' := AQ^{-1}$, we can assume
  \[
  B_*A_*=\left(\begin{array}{c}1 \\ 0 \end{array}\right).
  \]
  Hereafter, unless otherwise noted, we will continue to denote them as $(A, B)$.\\

  Let the realization parameters be $(A, B) = (\alpha, \beta)$, such that $\beta\alpha = B_*A_*$. Perform the coordinate transformation $B' := B - \beta$ and $A' := A - \alpha$, ensuring that the origin corresponds to $(\alpha, \beta)$. Noting that the ranks of $\beta, \alpha$ are at least 1, we may assume $\beta_{11} \neq 0$ and $\alpha_{11} \neq 0$. Then,
  \begin{align*}
    S =& (B+\beta)(A+\alpha)-B_*A_*\\
    =& \left(
      \begin{array}{cc} 
        b_{11}+\beta_{11} & b_{12}+\beta_{12}\\
        b_{21}+\beta_{21} & b_{22}+\beta_{22}
      \end{array}\right)
      \left(\begin{array}{c} 
        a_{11}+\alpha_{11}\\
        a_{21}+\alpha_{21}
      \end{array}\right)
      -\left(\begin{array}{c}
        1  \\ 
        0 
        \end{array}\right)
  \end{align*}
  can be expressed.

  Considering the $(1,1)$ component of $S$ near $b_{11}=0$ where $b_{11}+\beta_{11} \neq 0$, we find
  \begin{align*}
  &(b_{11}+\beta_{11})(a_{11}+\alpha_{11})+(b_{12}+\beta_{12})(a_{21}+\alpha_{21})-1=0\\
  \Leftrightarrow\ & a_{11}=-\frac{(b_{12}+\beta_{12})(a_{21}+\alpha_{21})-1}{b_{11}+\beta_{11}}-\alpha_{11}
  \end{align*}
  leads to the coordinate transformation
  \[
  a'_{11}:=a_{11}+\frac{(b_{12}+\beta_{12})(a_{21}+\alpha_{21})-1}{b_{11}+\beta_{11}}+\alpha_{11}
  \]
  resulting in the $(1,1)$ component of $S$ being
  \[
  (b_{11}+\beta_{11})\cdot a'_{11}.
  \]

  Next, for the $(2,1)$ component of $S$, noting $a_{11}=0$ where $a_{11}+\alpha_{11} \neq 0$, we get
  \[
  (b_{21}+\beta_{21})(a_{11}+\alpha_{11})+(b_{22}+\beta_{22})(a_{21}+\alpha_{21})=0
  \]
  \begin{align*}
    \Leftrightarrow\ b_{21} &=-\frac{(b_{22}+\beta_{22})(a_{21}+\alpha_{21})}{a_{11}+\alpha_{11}}-\beta_{21}\\
    &=-\frac{(b_{11}+\beta_{11})(b_{22}+\beta_{22})(a_{21}+\alpha_{21})}
    {(b_{11}+\beta_{11})a'_{11}-(b_{12}+\beta_{12})a_{21}-b_{12}\alpha_{21}+\beta_{11}\alpha_{11}}
    -\beta_{21}
  \end{align*}
  leads to the transformation
  \[
  b'_{21}:=
  b_{21}
  +\frac{(b_{11}+\beta_{11})(b_{22}+\beta_{22})(a_{21}+\alpha_{21})}
  {(b_{11}+\beta_{11})a'_{11}-(b_{12}+\beta_{12})a_{21}-b_{12}\alpha_{21}+\beta_{11}\alpha_{11}}
  +\beta_{21}
  \]
  resulting in the $(2,1)$ component of $S$ being
  \begin{align*}
  &b'_{21}\cdot(a_{11}+\alpha_{11})\\
  =&
  b'_{21}\cdot
  \frac{(b_{11}+\beta_{11})a'_{11}-(b_{12}+\beta_{12})a_{21}-b_{12}\alpha_{21}+\beta_{11}\alpha_{11}}
  {b_{11}+\beta_{11}}
  \end{align*}
  expressed.

  With these coordinate transformations,
  \[
  S=
  \left(\begin{array}{c}
    (b_{11}+\beta_{11})a_{11}^{\prime}  \\
    b'_{21}\cdot
  \frac{(b_{11}+\beta_{11})a'_{11}-(b_{12}+\beta_{12})a_{21}-b_{12}\alpha_{21}+\beta_{11}\alpha_{11}}
  {b_{11}+\beta_{11}} 
  \end{array}\right)
  \]
  is expressed, and it is clear that $S=0$ only when $(a'_{11}, b'_{21})=(0,0)$.
  \begin{align*}
  &\left.\frac{\partial f}{\partial a'_{11}}\right|_{(a_{11}^{\prime}, a_{21}, b_{11}, b_{12}, b_{21}^{\prime}, b_{22})=0}
    =-\beta_{11}z_1x_1\\
  &\left.\frac{\partial f}{\partial b'_{21}}\right|_{(a_{11}^{\prime}, a_{21}, b_{11}, b_{12}, b_{21}^{\prime}, b_{22})=0}
    =-\alpha_{11}z_2x_1
  \end{align*}
  are linearly independent. Therefore, treating other parameters as constants, this statistical model satisfies Assumption~\ref{mainass} at the origin $(a_{11}^{\prime}, a_{21}, b_{11}, b_{12}, b_{21}^{\prime}, b_{22})=0$ when $(d_1, r, m)=(2, 2, 1)$, and by Main Theorem~\ref{mainthm1}, the real log canonical threshold at the origin $O$ is
  \[
  \lambda_O = \frac{2}{2} = 1
  \]
  (multiplicity is 1).\\
\end{proof}

\section{Conclusion}
In this paper, we have further generalized the main results obtained in \cite{kurumadani1}. Specifically, we have generalized \cite[Main Theorem 1]{kurumadani1} and demonstrated that the formula for the Taylor expansion of the Kullback-Leibler divergence is valid even when some parameters $\tau$ are considered as constants (Main Theorem~\ref{mainthm}). Main Theorem~\ref{mainthm1} applies this to models satisfying Assumption~\ref{mainass} and provides a concrete formula for the real log canonical threshold.

As concrete examples using Main Theorem~\ref{mainthm1}, we have presented several calculations of the real log canonical threshold in the set of realization parameters $\Theta_*$, where the model is non-singular. These include the upper bound of the learning coefficient for mixed binomial distributions and the real log canonical threshold for specific cases of reduced rank regression models.

Our Main Theorem~\ref{mainthm1} is expected to be broadly applicable at non-singular points. That is, by computing the real log canonical threshold at non-singular points using Main Theorem~\ref{mainthm1}, it is possible to obtain an upper bound on the learning coefficient. Furthermore, finding 'better' non-singular points allows for a tighter evaluation of the upper bound on the learning coefficient. This upper bound is also tighter than those obtained in previous studies (\ref{eq:zyokai02}). This approach does not require an actual blow-up; it suffices to verify whether Assumption~\ref{mainass} is satisfied, thus reducing computational effort. This is considered an effective approach for obtaining upper bounds on the learning coefficient for all statistical models. Future research will focus on deriving the real log canonical thresholds at singular points within the set of realization parameters $\Theta_*$.

\backmatter

\bmhead{Acknowledgements}
I would like to express my gratitude to Professor Yuzuru Suzuki of Osaka University, who taught me the fundamentals of Bayesian theory and provided me with a research theme that bridges algebraic geometry and statistics, my area of expertise. I am also grateful for his feedback and review of this paper. Furthermore, I appreciate Professor Sumio Watanabe for proposing the concept of the learning coefficient.

\section*{Declarations}
The author has no relevant financial or non-financial interests to disclose.


    

\bibliography{sn-bibliography}
\end{document}